\documentclass[twoside]{article}

\usepackage{amsmath,amsthm,amssymb}
\usepackage{anyfontsize}
\def\frak{\mathfrak}
\usepackage{amsfonts,amssymb,mathrsfs,amscd}

\usepackage{newtxtext,newtxmath}

\usepackage[text={12.5cm,19cm},centering,paperwidth=17cm,paperheight=24cm]{geometry}

\usepackage[fontsize=10.4pt]{scrextend}

\def\titlerunning#1{\gdef\titrun{#1}}

\makeatletter
\def\author#1{\gdef\autrun{\def\and{\unskip, }#1}\gdef\@author{#1}}
\def\address#1{{\def\and{\\\hspace*{15.6pt}}\renewcommand{\thefootnote}{}\footnote{#1}}\markboth{\autrun}{\titrun}}
\makeatother

\def\email#1{email: \href{mailto:#1}{#1} }
\def\subjclass#1{\par\bigskip\noindent\textbf{Mathematics Subject Classification 2020.} #1}
\def\keywords#1{\par\smallskip\noindent\textbf{Keywords.} #1}

\newenvironment{dedication}{\itshape\center}{\par\medskip}
\newenvironment{acknowledgments}{\bigskip\small\noindent\textit{Acknowledgments.}}{\par}

\newtheorem{thm}{Theorem}[section]
\newtheorem{cor}[thm]{Corollary}
\newtheorem{lem}[thm]{Lemma}

\newtheorem{prop}[thm]{Proposition}

\theoremstyle{definition}
\newtheorem{defin}[thm]{Definition}

\newtheorem*{rem}{Remark}

\numberwithin{equation}{section}

\usepackage[hyperfootnotes=false,colorlinks=true,allcolors=blue]{hyperref}

\begin{document}

% Give an abbreviation of the title for the running page headers.
\titlerunning{Attractor  of the damped  2D Euler-Bardina equations}

% Here you can enter the full article title.
\title{\textbf{Sharp dimension  estimates of the attractor  of the damped  2D Euler-Bardina equations}}

% Here you can enter the full names of authors separated by \and.
\author{Alexei Ilyin \and Sergey Zelik}

% Please do not enter a date.
\date{}

\maketitle

% Here you can enter the address and email of each author separated by \and by following the example.
\address{A. Ilyin: Keldysh Institute of Applied Mathematics, Moscow, Russia; \email{ilyin@keldysh.ru} \and S.Zelik: Keldysh Institute of Applied Mathematics, Moscow, Russia, School of Mathematics and Statistics, Lanzhou
University, China  and University of Surrey, Department of
Mathematics, Guildford, GU2 7XH, United Kingdom; \email{s.zelik@surrey.ac.uk}}

% Here you can enter an optional dedication.
\begin{dedication}
To Ari Laptev on the occasion of his 70th birthday
\end{dedication}

% Here you can enter the abstract, MSC classes, and keywords.
\begin{abstract}
We prove  existence of the global attractor of the
damped and driven   2D Euler--Bardina equations
 on the torus and give an explicit two-sided estimate of its dimension
 that is sharp as $\alpha\to0^+$.
 \keywords{Damped
Euler-Bardina equations, $\alpha$ models, attractors, dimension
estimates}
\end{abstract}

\section{Introduction}

The Navier--Stokes system remains over the last decades  in the focus of the
theory of infinite dimensional dynamical systems (see, for example,
\cite{B-V, FMRT, Lad, S-Y, T} and the references therein).
For a system defined on a bounded two-dimensional domain it was shown that
the corresponding global attractor has finite fractal dimension.
The idea to use the  Lieb--Thirring inequalities \cite{Lieb}
for orthonormal families played an essential role
in deriving physically relevant upper bounds for the dimension. Furthermore,
the upper bounds in case of the torus $\mathbb{T}^2$ are
sharp up to a logarithmic correction \cite{Liu}.

Another model in incompressible hydrodynamics more recently studied
from the point of view of attractors is the two-dimensional damped Euler system
\begin{equation}\label{DEuler}
\left\{
  \begin{array}{ll}
    \partial_t u+( u,\nabla_x) u+\gamma u+\nabla_x p=g,\  \  \\
    \operatorname{div}u=0,\quad u(0)=u_0.
  \end{array}
\right.
\end{equation}
The linear damping term $\gamma u$ here makes the system dissipative
and is important in various geophysical models  \cite{Ped}.
The system  is studied either on a 2d manifold (torus, sphere) or
in  a bounded 2d domain with stress-free boundary conditions.
The natural phase space here is $H^1$ where
it easy to prove the existence of a solution of class $L^\infty(0,T, H^1)$
and dissipativity. However, the solution in this class is not known to be unique.

A weak attractor and a weak trajectory attractor  for \eqref{DEuler}
were constructed in \cite{IlyinEU} and \cite{CVrj}, respectively.
It was then shown in \cite{CIZ} and \cite{CVZ} that these
attractors are, in fact, compact in $H^1$ and the attraction
holds in the norm of $H^1$.

Closely related to \eqref{DEuler} is its  Navier--Stokes perturbation
\begin{equation}\label{DDNS}
\begin{cases}
\partial_t u+(u,\nabla_x)u+\nabla_x p+\gamma u=\nu \Delta_x u+g,\\
\operatorname{div} u=0,\quad u(0)=u_0,\ x\in \mathbb{T}^2,
\end{cases}
\end{equation}
which is studied in the vanishing viscosity limit $\nu\to0$. It is proved
in \cite{CIZ} that the attractors $\mathscr{A}_\nu$ of \eqref{DDNS}
tend in $H^1$ to the attractor $\mathscr{A}$ of \eqref{DEuler},
and, furthermore \cite{IMT}, their fractal dimension satisfies an order-sharp
(as $\nu\to0^+$) two sided estimate
\begin{equation}\label{two-nu}
1.5\cdot 10^{-6}\frac{\|\operatorname{curl}g\|_{L^2}^2}{\nu\gamma^3}\le
\dim\mathscr{A}_\nu\le\frac{3\pi}{128}
\frac{\|\operatorname{curl}g\|_{L^2}^2}{\nu\gamma^3},
\end{equation}
where the left-hand side estimate holds for a
specially chosen Kolmogorov forcing, and in the
right-hand side
we used the recent estimate of the Lieb--Thirring constant on the torus \cite{ZIL-JFA}.
This indicates (or at least suggests) that the problem of estimating the dimension
of the attractor $\mathscr{A}$ of \eqref{DEuler} is difficult and the attractor may well be
infinite dimensional.

In this work we use a different approximation of \eqref{DEuler},
namely, the so-called inviscid damped Euler--Bardina model
(see \cite{BFR80,Bardina, Titi-Varga} and the references therein)
\begin{equation}\label{DEalpha}
\left\{
  \begin{array}{ll}
    \partial_t u+(\bar u,\nabla_x)\bar u+\gamma u+\nabla_x p=g,\  \  \\
    \operatorname{div} u=0,\quad u(0)=u_0,\quad  u=(1-\alpha\Delta_x)\bar u,
  \end{array}
\right.
\end{equation}
on a 2D torus $\mathbb{T}^2=(0,L)^2$ with the forcing $g\in
H^1(\mathbb{T}^2)$ and  $\alpha,\gamma>0$. The system is studied  in the phase
space $H^{-1}(\mathbb{T}^2)\cap\{\operatorname{div} u=0\}=:H^{-1}$.
We also assume that $\int_{\mathbb{T}^2}( u,\bar u, g)dx=0$.
Here $\alpha=\alpha'L^2$ and $\alpha'>0$ is a small dimensionless parameter, so that
$\bar u$ is a smoothed (filtered) vector field.

 Alternatively, equations \eqref{DEalpha} can be considered as a particular case of
 the so-called Kelvin-Voight regularization of damped Navier-Stokes equations. Indeed, rewriting
 it in terms of the variable $\bar u$, we arrive at the equations
\begin{equation}\label{DEalpha1}
\left\{
  \begin{array}{ll}
    \partial_t\bar u-\alpha\Delta_x\partial_t\bar u+(\bar u,\nabla_x)\bar u+
 \gamma\alpha \bar u=\gamma\alpha\Delta_x\bar u+g,\  \  \\
    \operatorname{div}  \bar u=0,\quad \bar u(0)=\bar u_0,
  \end{array}
\right.
\end{equation}
which are damped Kelvin-Voight Navier-Stokes equations with the specific choice of Ekman damping
 parameter $\bar\gamma:=\gamma\alpha$ which coincides with the kinematic viscosity
  $\nu:=\gamma\alpha$, see \cite{Bardina,Titi-Varga} for more details. In the present paper
  we restrict our attention to 2D space periodic case  since only in this case the
  sharp lower bounds for the attractor's dimension can be obtained. The explicit upper bounds for this
  dimension in bounded and unbounded 3D domains will be given in the forthcoming paper \cite{IKZ}.

\medskip

Let us describe the results of this paper. In section~2  we
prove that system \eqref{DEalpha} is dissipative in the phase space $H^{-1}(\mathbb{T}^2)$
and prove the existence of the global attractor. Since for $\alpha>0$ the
convective  term is a bounded (compact) perturbation,
the existence of the global attractor is essentially an ODE result.

As far as the attractor is concerned, the choice of the phase
space can vary, however, in the next section~3 the phase space
$H^{-1}(\mathbb{T}^2)$ is most convenient for the estimates
of the global Lyapunov exponents by means
of the Constantin--Foias--Temam $N$-trace formula \cite{B-V, CF85, T}.

The following two-sided order-sharp  estimate (as $\alpha\to0^+$) of the dimension of the
global attractor $\mathscr {A}=\mathscr{A}_\alpha$ of system~\eqref{DEalpha}
is the main result of this work
\begin{equation}\label{alphaalpha}
6.46\cdot 10^{-7}\frac{\|\operatorname{curl} g\|^2_{L^2}}{\alpha\gamma^4}\le\dim_F\mathscr{A}\le
\frac{1}{8\pi}\frac{\|\operatorname{curl} g\|_{L^2}^2}{\alpha\gamma^4}\,.
\end{equation}
We observe that its looks somewhat similar to \eqref{two-nu} with
$\nu$ interchanged with $\alpha$.

While the right-hand side estimate is universal, the lower bound here holds as
in  \eqref{two-nu}  for a  family of Kolmogorov right-hand sides $g=g_s$
 specially chosen in section~4.
For this purpose we
use the instability analysis for the family
of generalized Kolmogorov flows generated by the family of right-hand sides
$$
g=g_s=\begin{cases}f_1=\frac1{\sqrt{2}\pi}\gamma\lambda(s)\sin sx_2,\\
g_2=0,\end{cases}
$$
where $s\in\mathbb{N}$. We show that the parameter $\lambda(s)$ can be chosen so that
the unstable manifold of the stationary solution corresponding to
$g_s$ has dimension $O(s^2)$, while for $s:=\alpha^{-1/2}$ we have
$$
\|\operatorname{curl} g_s\|_{L^2}=O(1)\quad\text{as}\ \alpha\to0.
$$
Since the unstable manifold lies in the global attractor,
it follows that we have a lower bound for the dimension of the  attractor $\mathscr{A}_s$
corresponding to $g=g_s$, which can be also written in an explicit form
given on the left-hand side in~\eqref{alphaalpha}.

Finally, in the Appendix in section~5 we prove the following
spectral inequality that is an essential analytical ingredient
in the proof of the upper bound for the dimension:
$$
\|\rho\|_{L^2}\le \frac1{2\sqrt{\pi}}m^{-1}n^{1/2},\qquad
\rho(x)=\sum_{j=1}^n|u_j(x)|^2,
$$
where
$$
u_j=(m^2-\Delta_x)^{-1/2}\psi_j,
$$
and $\{\psi_j\}_{j=1}^n$ is an orthonormal family in $L^2(\mathbb{T}^2)$
and each $\psi_j$ has mean value zero. The inequality of this type is proved (among others)
in \cite{LiebJFA} in the case of $\mathbb{R}^d$. The case of the torus requires a special treatment,
since in this case the parameter $m^2$ can no longer be scaled out.

\setcounter{equation}{0}
\section{Global attractor}\label{sec2}
%We consider the following modified damped Euler equations
%\begin{equation}\label{DEalpha}
%\left\{
%  \begin{array}{ll}
%    \partial_t u+(\bar u,\nabla_x)\bar u+\gamma u+\nabla_x p=g,\  \  \\
%    \operatorname{div} u=0,\quad u(0)=u_0,\quad  u=(1-\alpha\Delta_x)\bar u,
%  \end{array}
%\right.
%\end{equation}
%on a 2D torus $\mathbb{T}^2=(0,L)^2$ with the forcing $g\in
%H^1(\mathbb{T}^2)$ and parameters $\alpha,\gamma>0$  in the phase
%space $H^{-1}(\mathbb{T}^2)\cap\{\operatorname{div} u=0\}=:H^{-1}$.
%We also assume that $\int_{\mathbb{T}^2}( u,\bar u, g)dx=0$.
\begin{thm}\label{Th:exist}
There exist a unique solution $u=u(t)\in H^{-1}$
of system~\eqref{DEalpha}, which defines
the semigroup of  solution operators acting in $H^{-1}$:
$$
S(t)u_0=u(t), \quad S(0)u_0=u_0.
$$
\end{thm}
\begin{proof}
Since for $u\in H^{-1}$ we have $\bar u\in H^1$, it follows that
the nonlinear term $(\bar u,\nabla_x)\bar u$ is bounded and Lipschitz in
$H^{-1}$. Therefore the existence of a local in time unique
solution is straightforward. Furthermore, the solution is
continuous with respect to the  initial data $u_0\in H^{-1}$. The
fact that it exists for all $t>0$ follows from  a priori estimates
derived below.
\end{proof}

\noindent\textbf{Vorticity equation.} Taking $\operatorname{curl}$ from both sides of
\eqref{DEalpha} we arrive at the vorticity equation for $\omega=\operatorname{curl} u$
\begin{equation}\label{vort}
\partial_t\omega+(\bar u,\nabla_x)\bar\omega+\gamma\omega=\operatorname{curl} g,\ \omega=(1-\alpha\Delta_x)\bar\omega.
\end{equation}
Multiplying this equation by $\bar\omega$ and integrating over $x$, we get
$$
\aligned
\frac12\frac d{dt}\left(\|\bar \omega\|^2_{L^2}+\alpha\|\nabla_x\bar\omega\|^2_{L^2}\right)+
\gamma\left(\|\bar \omega\|^2_{L^2}+\alpha\|\nabla_x\bar\omega\|^2_{L^2}\right)=(\operatorname{curl} g,\bar\omega)=\\=
(g,\nabla_x^\perp\bar\omega)\le \|g\|_{L^2}\|\nabla_x\bar\omega\|_{L^2}\le\frac1{2\alpha\gamma}\|g\|_{L^2}^2+
\frac{\alpha\gamma}2\|\nabla_x\bar\omega\|_{L^2}^2.
\endaligned
$$
This gives by Gronwall's inequality the estimate
$$
\|\bar \omega(t)\|^2_{L^2}+\alpha\|\nabla_x\bar\omega(t)\|^2_{L^2}\le
\frac 1{\alpha\gamma^2}\|g\|^2_{L^2},
$$
which holds for every trajectory $u(t)$  as $t\to\infty$.

Alternatively, without integration by parts of the $\operatorname{curl}$ on the right-hand side,
we can write
$$
\|\operatorname{curl} g\|_{L^2}\|\bar\omega\|_{L^2}
\le\frac1{2\gamma}\|\operatorname{curl} g\|_{L^2}^2+
\frac{\gamma}2\|\bar\omega\|_{L^2}^2,
$$
giving as $t\to\infty$
$$
\|\bar \omega(t)\|^2_{L^2}+\alpha\|\nabla_x\bar\omega(t)\|^2_{L^2}\le \frac 1{\gamma^2}\|\operatorname{curl} g\|^2_{L^2},
$$
so that as $t\to\infty$
\begin{equation}\label{estonattr}
\|\bar \omega(t)\|^2_{L^2}+\alpha\|\nabla_x\bar\omega(t)\|^2_{L^2}\le
\frac 1{\gamma^2}\min\left(\frac{\|g\|_{L^2}^2}\alpha,\|\operatorname{curl} g\|^2_{L^2}\right)=:R_0^2.
\end{equation}
%This proves the existence of an absorbing ball in $H^{-1}$ for  $S(t)$.
\begin{prop}\label{Pr:absorb}
For any $R_1>0$ and any $u_0$ with $\|u_0\|_{H^{-1}}\le R_1$
there exists a time $T=T(R_1,R_0)$ such that for all $t\ge T(R_1,R_0)$
the solution $u(t)=S(t)u_0$ enters and never leaves
 the ball of radius $2R_0$ in $H^{-1}$:
$$
\|u(t)\|_{H^{-1}}\le 2R_0\quad \text{for}\quad t\ge T(R_1,R_0).
$$
In other words, this ball is an \textit{absorbing set} for
$S(t)$ in $H^{-1}$.
\end{prop}
\begin{proof}
The claim of the proposition follows from~\eqref{estonattr}
(where we may want to drop the term with $\alpha$ on the left-hand side) and
the following equivalence of the norms
$$
\|u\|_{H^{-1}}\sim \|\bar u\|_{H^1}=\|\nabla_x\bar u\|_{L^2}=\|\bar \omega\|_{L^2}.
$$
\end{proof}

\noindent\textbf{Asymptotic compactness and global attractor.}
To prove the asymptotic compactness of the semigroup $S(t)$ we construct
an attracting compact set in the absorbing ball in $H^{-1}$. For this purpose we
decompose $S(t)$ as follows
$$
S(t)=\Sigma(T)+S_2(t),\qquad S_2(t)=S(t)-\Sigma(t),
$$
where $\Sigma(t)$ is exponentially contacting and $S_2(t)$ is uniformly compact.
This decomposition (and its more elaborate variants) is very useful for dissipative hyperbolic problems
\cite{B-V, Ch-V-book, T} and the original idea goes back to \cite{Aro}.

Let $\Sigma(t)$ be the solution (semi)group of the linear
equation
$$
\partial_t v+\gamma v=0,\quad\operatorname{div} v=0, \quad v(0)=u_0,
$$
and $w$ (playing the role of $S_2(t)$) is the solution of
$$
\partial_t w+\gamma w+\nabla_x p=G(t):=-(\bar u,\nabla_x)\bar u+g,\quad\operatorname{div} w=0,\quad w(0)=0,
$$
where, of course, $u=v+w$. Obviously, $v(t)=e^{-\gamma t}u_0$
is exponentially decaying in $H^{-1}$ (and in every $H^s$). The right-hand side $G(t)$
in the ``linear'' equation for $w$ is bounded in
$H^{-\delta}$ for every $\delta>0$  uniformly in $t$.
In fact, since $\bar u(t)$ is bounded in $H^1$, it follows from
the Sobolev imbedding theorem and H\"older's inequality that
$(\bar u,\nabla_x)\bar u$ is bounded in $L^{2-\varepsilon}$ for $\varepsilon >0$,
while, in turn, by duality $L^{2-\varepsilon}\subset H^{-\delta}$,
$\delta=\delta(\varepsilon)=\varepsilon/(2-\varepsilon)$.

Taking into account that $w(0)=0$ we see that  the solution $w$ is bounded in $H^{-\delta}$ uniformly
with respect to $t$,  and since the imbedding
$H^{-\delta}(\mathbb{T}^2)\subset H^{-1}(\mathbb{T}^2)$
is compact, the asymptotic compactness of the solution semigroup
$S(t)$ is established.

We recall the definition of the global attractor.
\begin{defin}\label{Def:attr}
Let $S(t)$, $t\ge0$, be a semigroup  in a Banach space $H$.
The set $\mathscr A\subset H$ is a global attractor of the
semigroup $S(t)$ if
\par
1) The set $\mathscr A$ is compact in $H$.
\par
2) It is strictly invariant: $S(t)\mathscr A=\mathscr A$.
\par
3) It attracts the images of bounded sets in $H$ as $t\to\infty$, i.e.,
for every bounded set $B\subset H$ and every neighborhood
$\mathcal O(\mathscr A)$ of the set $\mathscr A$ in $H$
there exists $T=T(B,\mathcal O)$ such that for all $t\ge T$
$$
S(t)B\subset\mathcal O(\mathscr A).
$$
\end{defin}

The following general result  (see, for instance, \cite{B-V, Ch-V-book, T}) gives
sufficient conditions for the existence of a global attractor.%\cite{BV,Ch-V-book,Lad,S-Y,temam} for its proof.

\begin{thm}\label{Th:attr} Let $S(t)$ be a semigroup
in a Banach space $H$. Suppose that
\par
1) $S(t)$ possesses a bounded absorbing ball $B\subset H$;
\par
2) For every fixed $t\ge0$ the map $S(t): B\to  H$ is continuous.
\par
3) $S(t)$ is asymptotically compact.
\par

\par
Then the semigroup $S(t)$ possesses a global attractor
$\mathscr A\subset B$. Moreover, the attractor $\mathscr A$ has the following structure:
\begin{equation}\label{sections}
\mathscr A=\mathscr K\big|_{t=0},
\end{equation}
where $\mathscr K\subset L^\infty(\mathbb R, H)$ is the set of complete
trajectories $u:\mathbb R\to H$ of semigroup $S(t)$ which are defined
for all $t\in\mathbb R$ and bounded.
\end{thm}

We are now prepared to state the main result of this section, whose proof directly follows from
Theorem~\ref{Th:attr} since its assumptions were verified above.
\begin{thm}\label{Th:Attr-alpha}
The semigroup $S(t)$ corresponding to \eqref{DEalpha}
possesses in the phase space $H^{-1}$ a global attractor $\mathscr A$.
\end{thm}

\setcounter{equation}{0}
\section{Dimension estimate}\label{sec3}
\begin{thm}\label{Th:est}
The global attractor $\mathscr{A}$ corresponding to the  regularized damped Euler
system~\eqref{DEalpha} has finite fractal dimension
\begin{equation}\label{dim-est}
\dim_F\mathscr{A}\le
\frac{1}{8\pi}\frac{\|\operatorname{curl} g\|_{L^2}^2}{\alpha\gamma^4}\,.
\end{equation}
\end{thm}
\begin{proof}
 Note that the differentiability of the solution
semigroup $S(t): H^{-1}\to H^{-1}$ is obvious here, so we only need to estimate the traces.
The equation
 of variations for equation \eqref{DEalpha} reads:
 $$
 \partial_t\theta=-\gamma\theta-(\bar u,\nabla_x)\bar\theta-(\bar\theta,\nabla_x)\bar u-\nabla_x p=:L_{u(t)}\theta,
 \quad \operatorname{div} \theta=0.
$$
We will estimate volume contraction factor in the space $H^{-1}$ endowed with the norm
$$
\|\theta\|_{\alpha}^2:=\|\bar\theta\|^2_{L^2}+\alpha\|\nabla_x\bar\theta\|^2_{L^2}=
\|\bar\theta\|^2_{L^2}+\alpha\|\operatorname{curl}\bar\theta\|^2_{L^2}=\|(1-\alpha\Delta_x)^{-1/2}\theta\|^2
$$
and scalar product
$$
(\theta,\xi)_\alpha=((1-\alpha\Delta_x)^{-1/2}\theta,(1-\alpha\Delta_x)^{-1/2}\xi).
$$
The numbers $q(n)$ (the sums of the first $n$ global Lyapunov exponents) are
estimated/defined as follows \cite{T}
$$
q(n)\le(:=)\limsup_{t\to\infty}\sup_{u(t)\in\mathscr{A}}\sup_{\{\theta_j\}_{j=1}^n}\frac1t\int_0^t
\sum_{j=1}^n(L_{u(\tau)}\theta_j,\theta_j)_\alpha d\tau,
$$
where $\{\theta_j\}_{j=1}^n$ is an orthonormal family with respect to
$(\cdot,\cdot)_\alpha$:
\begin{equation}\label{alpha-orth}
(\theta_i,\theta_j)_\alpha=\delta_{i\,j},\quad\operatorname{div} \theta_j=0
\end{equation}
and $u(t)$ is an arbitrary trajectory on the attractor.

Then,
$$
\aligned
\sum_{j=1}^n(L_{u(t)}\theta_j,\theta_j)_\alpha=\sum_{j=1}^n(L_{u(t)}\theta_j,\bar\theta_j)=\\=
-\sum_{j=1}^n\gamma\|\theta_j\|^2_{\alpha}-\sum_{j=1}^n((\bar\theta_j,\nabla_x)\bar u,\bar\theta_j)=\\=
-\gamma n-\sum_{j=1}^n((\bar\theta_j,\nabla_x)\bar u,\bar\theta_j).
\endaligned
$$
Next,  $\theta_j\in H^{-1}$ are such that the family
$\{(1-\alpha\Delta_x)^{-1/2}\theta_j\}_{j=1}^n=:\{\psi_j\}_{j=1}^n$ is orthonormal in $L^2$.
By definition
$$
\bar\theta_j=(1-\alpha\Delta_x)^{-1}\theta_j=(1-\alpha\Delta_x)^{-1/2}\psi_j,
$$
and therefore for the function
$$
\rho(x)=\sum_{j=1}^n|\bar\theta_j(x)|^2
$$
in view of \eqref{Lieb-bound2} (with $p=2$) and \eqref{Lieb-bound22} we have the estimate
\begin{equation}\label{Lieb-for-us}
\|\rho\|_{L^2}\le\mathrm{B}_2\frac{n^{1/2}}{\sqrt{\alpha}},\quad
\mathrm{B}_2\le\frac1{2\sqrt{\pi}}.
\end{equation}
Since $\operatorname{div} \bar u=0$, we have pointwise
$$
|(\bar\theta,\nabla_x)\bar u\cdot\bar\theta|\le\frac1{\sqrt{2}}|\bar\theta(x)|^2|\nabla_x \bar u(x)|
$$
and using \eqref{Lieb-for-us} and \eqref{estonattr} we obtain
$$\aligned
\sum_{j=1}^n|((\bar\theta_j,\nabla_x)\bar u,\bar\theta_j)|\le\frac1{\sqrt{2}}
\int_{\mathbb{T}^2}\rho(x)|\nabla_x \bar u|dx\le\\\le
\frac1{\sqrt{2}}\|\rho\|_{L^2}\|\nabla_x\bar u\|_{L^2}=
\frac1{\sqrt{2}}\|\rho\|_{L^2}\|\bar \omega\|_{L^2}\le\\\le
\frac{\mathrm{B}_2}{\sqrt{2}}\frac{n^{1/2}}{\sqrt{\alpha}}\frac{\|\operatorname{curl} g\|_{L^2}}{\gamma}.
\endaligned
$$

Since all our estimates are uniform with respect time and the $\alpha$-orthonormal family
\eqref{alpha-orth}, it follows that
$$
q(n)\le-\gamma n +\frac{\mathrm{B}_2}{\sqrt{2}}\frac{n^{1/2}}{\sqrt{\alpha}}\frac{\|\operatorname{curl} g\|_{L^2}}{\gamma}.
$$
The number $n^*$ such that for which $q(n^*)=0$  and $q(n)<0$ for $n>n^*$ is the upper bound both
for the Hausdorff \cite{B-V,T} and the fractal \cite{Ch-I2001,Ch-I}
dimension of the global attractor $\mathscr{A}$:
$$
\dim_F\mathscr{A}\le \frac{\mathrm{B}_2^2}2\frac{\|\operatorname{curl} g\|_{L^2}^2}{\alpha\gamma^4}\le
\frac{1}{8\pi}\frac{\|\operatorname{curl} g\|_{L^2}^2}{\alpha\gamma^4}\,.
$$
\end{proof}

\setcounter{equation}{0}
\section{A sharp lower bound}\label{sec4}

In this section we derive a sharp lower bound
for the dimension of the global attractor $\mathscr{A}$
based on the generalized Kolmogorov flows \cite{MS,Liu,Y}.

We consider the vorticity equation \eqref{vort}
\begin{equation}\label{vort1}
\partial_t\omega+J\bigl((\Delta_x-\alpha\Delta_x^2)^{-1}\omega,(1-\alpha\Delta_x)^{-1}\omega\bigr)+\gamma\omega=\operatorname{curl} g,
\end{equation}
where $J$ is the Jacobian operator
$$
J(a,b)=\nabla^\perp a\cdot\nabla b=\partial_{x_1}a\,\partial_{x_2}b-
\partial_{x_2} a\,\partial_{x_1} b.
$$
Next, let $g=g_s$ be a family of right-hand sides
\begin{equation}\label{Kolmf}
g=g_s=\begin{cases}f_1=\frac1{\sqrt{2}\pi}\gamma\lambda(s)\sin sx_2,\\
g_2=0,\end{cases}
\end{equation}
where $s\in\mathbb{N}$, and $\lambda(s)$ is a parameter to be
chosen later. Then
\begin{equation}\label{rotf}
\operatorname{curl} g_s=-\frac1{\sqrt{2}\pi}\gamma\lambda(s)s\cos sx_2,
\end{equation}
and
\begin{equation}\label{normf}
\|g_s\|_{L^2}^2=\gamma^2\lambda(s)^2,\qquad \|\operatorname{curl} g_s\|^2_{L^2}=\gamma^2\lambda(s)^2 s^2.
\end{equation}

Corresponding to the family \eqref{rotf} is the family of stationary solutions
\begin{equation}\label{omega-s}
\omega_s=-\frac1{\sqrt{2}\pi}\lambda(s)s\cos sx_2
\end{equation}
of equation \eqref{vort1}.
Since $\omega_s$ depends only on $x_2$, the nonlinear term vanishes:
$$
J\bigl((\Delta_x-\alpha\Delta_x^2)^{-1}\omega_s,
(1-\alpha\Delta_x)^{-1}\omega_s\bigr)\equiv0.$$

We linearize \eqref{vort1} about the stationary solution \eqref{omega-s} and
consider the eigenvalue problem
\begin{equation}\label{eig}
    \aligned{\mathcal L}_{s}\omega:=
    &J\bigl((\Delta_x-\alpha\Delta_x^2)^{-1}\omega_s,(1-\alpha\Delta_x)^{-1}\omega\bigr)+\\+
&J\bigl((\Delta_x-\alpha\Delta_x^2)^{-1}\omega,(1-\alpha\Delta_x)^{-1}\omega_s\bigr)+\gamma\omega=
-\sigma\omega.
\endaligned
\end{equation}
The solutions with $\mathrm{Re}\sigma>0$ will be unstable eigenmodes.

We use the orthonormal basis of trigonometric functions,
 which are the eigenfunctions of the Laplacian,
$$
\aligned
\left\{\frac1{\sqrt{2}\pi}\sin kx,\ \frac1{\sqrt{2}\pi}\cos kx\right\},
\quad kx=k_1x_1+k_2x_2,\\
\quad k\in \mathbb{Z}^2_+=\{k\in\mathbb{Z}^2_0,\ k_1\ge0, k_2\ge 0\}
\cup\{k\in\mathbb{Z}^2_0,\ k_1\ge1, k_2\le 0\}
\endaligned
$$
and write
$$
\omega(x)=\frac1{\sqrt{2}\pi}\sum_{k\in\mathbb{Z}_+^2}a_k\cos kx+ b_k\sin kx.
$$
Substituting this into~(\ref{eig}) and using the equality
$J(a,b)=-J(b,a)$ we obtain
\begin{equation}\label{eigseries}
\aligned
\frac{s\lambda(s)}{\sqrt{2}\pi(s^2+\alpha s^4)}\sum_{k\in\mathbb{Z}_+^2}
\left(\frac{k^2-s^2}{k^2+\alpha k^4}\right)
J(&\cos sx_2, a_k\cos kx+b_k\sin kx)+\\+(\gamma+\sigma)
\sum_{k\in\mathbb{Z}_+^2}(&a_k\cos kx+b_k\sin kx)=0.
\endaligned
\end{equation}
 Next, we have the following two
similar formulas
$$
\aligned
J(\cos sx_2,&\cos(k_1x_1+k_2x_2))=-k_1 s\sin sx_2\sin(k_1x_1+k_2x_2)=\\=
&\frac{k_1s}2(\cos(k_1x_1+(k_2+s)x_2)-\cos(k_1x_1+(k_2-s)x_2),\\
J(\cos sx_2,&\sin(k_1x_1+k_2x_2))=k_1 s\sin sx_2\cos(k_1x_1+k_2x_2)=\\=
&\frac{k_1s}2(\sin(k_1x_1+(k_2+s)x_2)-\sin(k_1x_1+(k_2-s)x_2),
\endaligned
$$
which we substitute  into (\ref{eigseries}) and collect the terms
 with $\cos(k_1x_1+k_2x_2)$. We obtain the
following equation for the coefficients $a_{k_1,k_2}$(the equation
for $b_{k_1,k_2}$ is exactly the same):
$$
\aligned
-\Lambda(s) k_1\left(
\frac{k_1^2+(k_2+s)^2-s^2}
{k_1^2+(k_2+s)^2+\alpha(k_1^2+(k_2+s)^2)^2)}
\right)\,&a_{k_1\,k_2+s}+\\
+\Lambda(s) k_1\left(\frac{k_1^2+(k_2-s)^2-s^2}
{k_1^2+(k_2-s)^2+\alpha(k_1^2+(k_2-s)^2)^2)}
\right)\,&a_{k_1\,k_2-s}+
(\gamma\!+\!\sigma)a_{k_1\,k_2}=0,
\endaligned
$$
where
\begin{equation}\label{lambdaLmabda}
\Lambda=\Lambda(s):=\frac{s^2\lambda(s)}{2\sqrt{2}\pi(s^2+\alpha s^4)}=
\frac{\lambda(s)}{2\sqrt{2}\pi(1+\alpha s^2)}\,.
\end{equation}
 We set here
$$
a_{k_1\,k_2}\left(
\frac{k^2-s^2}
{k^2+\alpha k^4}
\right)=:c_{k_1\,k_2},
$$
and
setting further
$$
\aligned
k_1=t,\quad k_2=sn+r,\quad\text{and}\quad c_{t\, sn+r}=e_n,\\
t=1,2,\dots\,,\quad r\in\mathbb{Z},\quad r_{\min}<r<r_{\max},
\endaligned
$$
where the numbers $r_{\min},r_{\max}$ satisfy $r_{\max}-r_{\min}<s$ and will
be specified below
we obtain for each $t$ and $r$ the following three-term recurrence relation:
\begin{equation}\label{rec}
d_ne_n+e_{n-1}-e_{n+1}=0, \qquad n=0,\pm1,\pm2,\dots\, ,
\end{equation}
where
\begin{equation}\label{dn}
d_n=\frac
{\left(t^2+(sn+r)^2+\alpha(t^2+(sn+r)^2)^2\right)(\gamma+\sigma)}
{\Lambda t(t^2+(sn+r)^2-s^2)}\,.
\end{equation}

We look for non-trivial decaying solutions $\{e_n\}$ of (\ref{rec}), (\ref{dn}).
Each non-trivial decaying solution with $\mathrm{Re}\,\sigma>0$
produces an unstable eigenfunction $\omega$ of the
eigenvalue problem~(\ref{eig}).

\begin{thm}\label{Th:unstable}
Given an integer $s>0$
let a fixed pair of integers $t,r$ belong to a bounded region
$A(\delta)$ defined by conditions
\begin{equation}\label{cond}
\aligned
t^2+r^2<s^2/3,\quad t^2+(-s+r)^2>s^2,\quad t^2+(s+r)^2>s^2, \quad t\ge\delta s,\\
%t\in\mathbb{N},\quad r\in\Zbb,
\quad r_{\min}<r<r_{\max},
\quad r_{\min}=-\frac s{6},\ r_{\max}=\frac s6,
%\quad\text{\rm for}
\quad0<\delta<1/\sqrt{3}.
\endaligned
\end{equation}
 For any $\Lambda>0$ there exists a unique real eigenvalue
$\sigma=\sigma(\Lambda)$, which increases monotonically
as $\Lambda\to\infty$ and satisfies the inequality
\begin{equation}\label{olambda}
 \Lambda \frac{ 21\sqrt{2} \delta^2 s}{55(1+\alpha s^2)}-\gamma\le
 \sigma(\Lambda)\le
    \Lambda\frac{\sqrt{2} \delta^{-1}s}{(1+\alpha s^2)}-\gamma.
\end{equation}
The unique $\Lambda_0=\Lambda_0(s)$ solving the equation
$$
\sigma(\Lambda_0)=0
$$
satisfies the two-sided estimate
\begin{equation}\label{lambdaots}
\frac{\gamma\delta}{\sqrt{2}}\frac{(1+\alpha s^2)}{s}<\Lambda_0<
\frac{55\gamma\delta^{-2}}{21\sqrt{2}}\frac{1+\alpha s^2}s\,.
\end{equation}
\end{thm}

\begin{proof}
We shall follow quite closely the argument in \cite{IT1} and  first
observe that the following inequalities hold for any~$(t,r)$
satisfying~(\ref{cond}):
\begin{equation}\label{tr}
\aligned
&s^2\le t^2+(-s+r)^2=\text{\rm{dist}}((0,s),(t,r))^2\le
\text{\rm dist}((0,s),C)^2=(5/3)s^2,\\
&s^2\le t^2+(s+r)^2=\text{\rm{dist}}((0,-s),(t,r))^2\le
\text{\rm dist}((0,-s),B)^2=(5/3)s^2,
\endaligned
\end{equation}
 where $B=(\sqrt{11}s/6,s/6)$ and
$C=(\sqrt{11}s/6,-s/6)$.

In view of (\ref{cond}) for any real $\sigma$ satisfying
$\sigma>-\gamma$ we have in
 (\ref{rec}), (\ref{dn})
\begin{equation}\label{condRe}
 d_0<0,\ d_n>0\quad \text{for}\quad  n\ne0\quad\text{and}\quad
\lim_{|n|\to\infty} d_n=\infty.
\end{equation}
The main tool in the analysis of ~(\ref{rec}) are
 continued fractions and a variant of Pincherle's theorem
(see \cite{JT}, \cite{Liu}, \cite{MS}, \cite{Y})
saying that under condition~(\ref{condRe}) recurrence relation~(\ref{rec})
has a decaying solution $\{e_n\}$ with $\lim_{|n|\to\infty}e_n=0$ if and only if
\begin{equation}\label{cont}
-d_0=\cfrac{1}{
       d_{-1}+\cfrac{1}{
         d_{-2}+\dots}}\
         +\
         \cfrac{1}{
       d_{1}+\cfrac{1}{
         d_{2}+\dots}}\ .
\end{equation}
Next, we set
\begin{equation}\label{f}
f(\sigma)=-d_0=\frac
{(\gamma+\sigma)(t^2+r^2+\alpha(t^2+r^2)^2)}
{\Lambda t(s^2-(t^2+r^2))}\ ,
\end{equation}
\begin{equation}\label{g}
g(\sigma)=\cfrac{1}{
       d_{-1}+\cfrac{1}{
         d_{-2}+\dots}}\
         +\
         \cfrac{1}{
       d_{1}+\cfrac{1}{
         d_{2}+\dots}}\ .
\end{equation}
It follows from (\ref{f}) that
$$
f(-\gamma)=0,\quad\text{and}\quad f(\sigma)\to\infty\quad\text{as}
\quad \sigma\to\infty,
$$
while  (\ref{g}) and  (\ref{dn}) give that
$$
g(\sigma)<\frac1{d_{-1}}+\frac1{d_{1}}\,,\quad
g(\sigma)\to 0\quad\text{as}\quad \sigma\to\infty.
$$
Hence, there exists a $\sigma>-\gamma$ such that
\begin{equation}\label{f=g}
f(\sigma)=g(\sigma).
\end{equation}
From elementary properties of continued fractions we deduce as
in~\cite{Liu}, \cite{Y}
 that  $\sigma=\sigma(\Lambda)$ so obtained is unique and
increases monotonically with $\Lambda$.

To establish~(\ref{olambda}) we deduce from (\ref{f=g}) and~(\ref{g}) that
\begin{equation}\label{between}
         \cfrac{1}{
       d_{-1}+\cfrac{1}{
         d_{-2}}}\
         +\
         \cfrac{1}{
       d_{1}+\cfrac{1}{
         d_{2}}}\
         <
         f(\sigma)<\frac{1}{d_{-1}}+\frac{1}{d_{1}}\ .
\end{equation}
Using~\eqref{cond} we see from \eqref{dn} that
$$
\aligned
\frac1{d_{\pm 1}}=\frac{\Lambda  t}{\gamma+\sigma}\cdot
\frac{t^2+(s\pm r)^2-s^2}{t^2+(s\pm r)^2+\alpha(t^2+(s\pm r)^2)^2)}\le\\\le
\frac{\Lambda  t}{\gamma+\sigma}\cdot
\frac{1}{1+\alpha(t^2+(s\pm r)^2)}\le
\frac{\Lambda  t}{\gamma+\sigma}\cdot
\frac{1}{1+\alpha s^2}
\endaligned
$$
and therefore from  the right-hand side inequality in \eqref{between}
it follows that
$$
f(\sigma)=
\frac
{(\gamma+\sigma)(t^2+r^2+\alpha(t^2+r^2)^2)}
{\Lambda t(s^2-(t^2+r^2))}<\frac{1}{d_{-1}}+\frac{1}{d_{1}}<
\frac{2\Lambda t}{\gamma+\sigma}\cdot
\frac{1}{1+\alpha s^2}\,,
$$
which gives the right-hand side inequality in~\eqref{olambda}:
$$
(1+\alpha s^2)(\gamma+\sigma)^2\le\frac{2\Lambda^2t^2(s^2-(t^2+r^2))}{t^2+r^2+\alpha(t^2+r^2)^2}
\le\frac{2\Lambda^2t^2s^2}{t^2+\alpha t^4}=
\frac{2\Lambda^2s^2}{1+\alpha t^2}\le\frac{2\Lambda^2\delta^{-2}s^2}{1+\alpha s^2}\,.
$$
From the left-hand side inequality in~(\ref{between}), where
$d_{-1},d_1,d_{-2},d_2,f>0$, we see that
\begin{equation}\label{two}
fd_{1}+\frac{f}{d_{2}}>1\quad\text{and}\quad fd_{-1}+\frac{f}{d_{-2}}>1.
\end{equation}
Next, it follows from~\eqref{cond} that $4sr\le2s^2/3$ and
$|2s+r|\ge 11 s/6$. Therefore
$$
\aligned
&\frac f{d_{2}}=\frac{(t^2+r^2+\alpha(t^2+r^2)^2)(t^2+(2s+r)^2-s^2)}
{(s^2-(t^2+r^2))(t^2+(2s+r)^2+\alpha(t^2+(2s+r)^2)^2)}<\\<&
\frac{(s^2/3+\alpha s^4/9)4s^2}{2s^2/3((11/6)^2s^2+\alpha(11/6)^4s^4)}=
\frac{2(1+t/3)}{(11/6)^2+t(11/6)^4}_{t=\alpha s^2}\le\frac{72}{121}.
\endaligned
$$
Along with \eqref{two} this implies that $fd_1>49/121$, which for
$r\ge0$ gives that
%$$
%\sigma(\Lambda)\ge c_1(\alpha,\gamma,\delta,t,r,s)\lambda,
%$$
$$
\aligned
&\frac{49}{121}<fd_1=\\=&\frac{(\gamma+\sigma)^2}{\Lambda^2}
\frac{(t^2+r^2+\alpha(t^2+r^2)^2)(t^2+(s+r)^2+\alpha(t^2+(s+r)^2)^2)}
{t^2(s^2-(t^2+r^2))(t^2+(s+r)^2-s^2)}<\\<&
\frac{(\gamma+\sigma)^2}{\Lambda^2}
\frac{(s^2/3+\alpha s^4/9)(5s^2/3+\alpha  s^4 25/9)}
{t^2(2/3)s^2t^2}<\\<&\frac{25}{18}
\frac{(\gamma+\sigma)^2}{\Lambda^2}
\frac{(1+\alpha s^2)^2}
{\delta^4s^2}\,,
\endaligned
$$
and  proves the left-hand side inequality in~(\ref{olambda}).
For $r<0$   we use $d_{-1}$ instead of $d_1$.

Finally, estimate \eqref{lambdaots} follows from~(\ref{olambda}) with $\sigma=0$.
\end{proof}

This result has the following important implications for the attractors of
the damped doubly regularized Euler equations \eqref{DEalpha}. Namely, it says that
estimate \eqref{dim-est} is order-sharp
in the limit as $\alpha\to0^+$.
\begin{cor}\label{Cor:lower}
The parameter $\lambda(s)$ in the family \eqref{Kolmf} can be
chosen so that the dimension of the corresponding global
attractor $\mathscr{A}=\mathscr{A}_s$ satisfies
$$
\dim\mathscr{A}\ge\mathrm{c}_1
\frac{\|\operatorname{curl} g\|^2_{L^2}}{\alpha\gamma^4}\,\quad
\mathrm{c}_1>6.46\cdot 10^{-7}.
$$
\end{cor}
\begin{proof}
Writing \eqref{lambdaots} in terms of $\lambda(s)$ (see \eqref{lambdaLmabda})
we see that for
$$
\lambda(s)=\frac{110\pi}{21}\gamma\delta^{-2}\frac{(1+\alpha s^2)^2}s
$$
each point in  $(t,r)$-plane satisfying (\ref{cond})
produces an unstable (positive) eigenvalue $\sigma$
of multiplicity two (the equation for the coefficients $b_k$ is
the same).
 Denoting by $d(s)$ the number of points of the integer
lattice inside the region $A(\delta)$ we  obviously have
\begin{equation*}\label{area}
  d(s):=  \#\{(t,r)\in D(s)=\mathbb{Z}^2\cap A(\delta)\}\backsimeq
   a(\delta)\cdot s^2\quad\text{\rm as}\quad s\to\infty,
\end{equation*}
where $a(\delta)\cdot s^2=|A(\delta)|$ is the area of the region
$A(\delta)$ (see Fig.1 in \cite{IT1}). Therefore the dimension of the unstable manifold
around the stationary solution $\omega_s$ in~\eqref{omega-s} is at
least $2a(\delta)\cdot s^2$. The solutions lying
on this unstable manifold are bounded on the entire time axis
$t\in \mathbb{R}$ , since they tend to $\omega_s$ as $t\to-\infty$
and all solutions are bounded as $t\to\infty$ in view of~\eqref{estonattr}.
Therefore representation~\eqref{sections} implies that \cite{B-V}

\begin{equation}\label{dim>}
    \dim{ A}\ge 2 d(s)\backsimeq 2a(\delta)\cdot s^2.
\end{equation}

It remains to express this lower bound in terms of the physical
parameters of the system. So far $s$ was an arbitrary (large) parameter.
We now set
$$s:=\frac1{\sqrt{\alpha}},$$
and see from \eqref{normf} that
$$
\|\operatorname{curl} g_s\|^2_{L^2}=\gamma^2\lambda(s)^2 s^2=
\left(\frac{110\pi}{21}\right)^2\gamma^4\delta^{-4}(1+\alpha s^2)^4=
\left(\frac{110\pi}{21}\right)^2\gamma^4\frac{2^4}{\delta^{4}}\,,
$$
and
$$
\frac{\|\operatorname{curl} g_s\|^2_{L^2}}{\alpha\gamma^4}=\frac1\alpha
\left(\frac{110\pi}{21}\right)^2\frac{2^4}{\delta^4}.
$$
We finally obtain  that \eqref{dim>} can be written in the form
$$
 \dim{\mathscr A}\ge
 \max_{0<\delta<1/\sqrt{3}}a(\delta)\delta^4\frac18\left(\frac{21}{110\pi}\right)^2\!
\frac{\|\operatorname{curl} g_s\|^2_{L^2}}{\alpha\gamma^4}=
6.46\cdot 10^{-7}\frac{\|\operatorname{curl} g_s\|^2_{L^2}}{\alpha\gamma^4}\,,
$$
where we optimized the result with respect to $\delta\in(0,3^{-1/2})$
\end{proof}

\setcounter{equation}{0}
\section{Appendix}\label{sec5}
We first recall a result in \cite{LiebJFA}, which we formulate for
the torus $\mathbb{T}^2$.
\begin{thm}\label{Th:Lieb}
Let a family $\{\psi_i\}_{i=1}^n$ be orthonormal in
$L^2(\mathbb{T}^2)$ and let $\int_{\mathbb{T}^2}\psi_i(x)dx=0$. Then the function
$$
\rho(x)=\sum_{i=1}^n|u_i(x)|^2,
$$
where
$$
u_i=(m^2-\Delta_x)^{-1/2}\psi_i,
$$
satisfies for $1\le p<\infty$ the estimate
\begin{equation}\label{Lieb-bound}
\|\rho\|_{L^p}\le \mathrm{B}_pm^{-2/p}n^{1/p}.
\end{equation}
\end{thm}

In our notation this estimate reads as follows.

\begin{cor}\label{C:Lieb}
Let a family $\{\psi_i\}_{i=1}^n$ be orthonormal in
$L^2(\mathbb{T}^2)$. Then the function
$$
\rho(x)=\sum_{i=1}^n|\bar\theta_i(x)|^2,
$$
where
$$
\bar\theta_i=(1-\alpha\Delta_x)^{-1/2}\psi_i,
$$
satisfies for $1\le p<\infty$ the estimate
\begin{equation}\label{Lieb-bound2}
\|\rho\|_{L^p}\le \mathrm{B}_p\alpha^{1/p\,-1}n^{1/p}.
\end{equation}
\end{cor}
\begin{proof}
We write
$$
\bar\theta_i=(1-\alpha\Delta_x)^{-1/2}\psi_i=\left(\frac1\alpha-\Delta_x\right)^{-1/2}\frac{\psi_i}{\sqrt{\alpha}},
$$
and \eqref{Lieb-bound} gives \eqref{Lieb-bound2}:
$$
\biggl\|\alpha\sum_{i=1}^n|\bar\theta_i|^2\biggr\|_{L^p}\le
\mathrm{B}_p\alpha^{1/p}n^{1/p}.
$$
\end{proof}

Next, we prove Theorem~\ref{Th:Lieb} for $p=2$ and give
the estimate of the constant.
\begin{prop}\label{Pr:Lieb}
Under the assumptions of Theorem~\ref{Th:Lieb} it holds for $p=2$
\begin{equation}\label{Lieb-bound22}
\|\rho\|_{L^2}\le \mathrm{B}_2m^{-1}n^{1/2},\qquad\mathrm{B}_2\le \frac1{2\sqrt{\pi}}.
\end{equation}
\end{prop}
\begin{proof}
We follow the idea in \cite{LiebJFA}. The main technical difference being that
in the discrete case we cannot scale out $m$ and therefore
we have to estimate the corresponding series over $\mathbb{Z}^2_0$ for all
$m>0$. This can be done for all $1\le p<\infty$, but each case requires a
special treatment (at least in our proof).

For a non-negative function $V=V(x)\in L^\infty$ we set
$$
H=V^{1/2}(m^2-\Delta_x)^{-1/2},\quad H^*=(m^2-\Delta_x)^{-1/2}V^{1/2}.
$$
We further set
$K=H^*H$ and claim that
\begin{equation}\label{trace}
\operatorname{Tr} K^2\le\frac1{4\pi}\frac1{m^2}\|V\|^2_{L^2}.
\end{equation}
In fact,
$$
\aligned
\operatorname{Tr} K^2=&\operatorname{Tr}\left((m^2-\Delta_x)^{-1/2}V(m^2-\Delta_x)^{-1/2}\right)^2\le\\&\le
\operatorname{Tr}\left((m^2-\Delta_x)^{-1}V^2(m^2-\Delta_x)^{-1}\right)=\\&=
\operatorname{Tr}\left(V^2(m^2-\Delta_x)^{-2}\right),
\endaligned
$$
where we used
 the Araki--Lieb--Thirring inequality for traces \cite{Araki,lthbook}:
$$
\operatorname{Tr}(BA^2B)^p\le\operatorname{Tr}(B^pA^{2p}B^p),
$$
and the cyclicity property of the trace.
Using the basis of orthonormal eigenfunctions of the Laplacian
$(2\pi)^{-1}e^{ikx}$, $k\in\mathbb{Z}^2_0=\mathbb{Z}^2\setminus\{0\}$
in view of~\eqref{ineqT2} below we find that
\begin{equation}\label{KSS}
\aligned
\operatorname{Tr} K^2\le&
\operatorname{Tr}\left(V^2(m^2-\Delta_x)^{-2}\right)=\\=&\frac1{4\pi^2}\sum_{k\in\mathbb{Z}^2_0}\frac1{(|k|^2+m^2)^2}
\int_{\mathbb{T}^2}V^2(x)dx\le\frac1{4\pi}\frac1{m^2}\|V\|_{L^2}^2.
\endaligned
\end{equation}
We can now complete the proof as in~\cite{LiebJFA}. We observe that
$$
\int_{\mathbb{T}^2}\rho(x)V(x)dx=\sum_{i=1}^n\|H\psi_i\|^2_{L^2},
$$
and in view of orthonormality and the variational principle
$$
\sum_{i=1}^n\|H\psi_i\|^2_{L^2}\le\sum_{i=1}^n\lambda_i,
$$
where $\lambda_i$ are the eigenvalues of the self-adjoint compact
operator $K$. This finally gives
$$
\aligned
\int_{\mathbb{T}^2}\rho(x)V(x)dx&\le\sum_{i=1}^n\lambda_i\le
n^{1/2}\left(\sum_{i=1}^n\lambda_i^2\right)^{1/2}\le\\&\le
n^{1/2}\left(\operatorname{Tr} K^2\right)^{1/2}\le
\frac{n^{1/2}m^{-1}}{2\sqrt{\pi}}\|V\|_{L^2}.
\endaligned
$$
Setting $V(x):=\rho(x)$ we complete the proof of
\eqref{Lieb-bound22}.
\end{proof}

\begin{rem}\label{R:vector}
{\rm
In fact, we proved Proposition~\ref{Pr:Lieb} in the scalar case, while
$\theta_j$ in the proof of Theorem~\ref{Th:est} are vector functions with
mean value zero and $\operatorname{div} \theta_j=0$. The result with the same
constant  and the same proof still holds in this case, since on the torus
the Helmgoltz--Leray projection commutes with the Laplacian
and the orthonormal family of vector-valued eigenfunctions
of the Stokes operator $v_k(x)=\frac1{2\pi}\frac{k^\perp}{|k|}e^{ikx}$,
$k\in\mathbb{Z}^2_0$ satisfy $|v_k(x)|=1/(2\pi)$ as in the scalar case
(see~\eqref{KSS}).
}
\end{rem}

\begin{rem}
{\rm
Inequality~\eqref{KSS} is nothing more than a special case of the Kato--Seiler--Simon inequality
\cite{traceSimon, lthbook}
$$
\|a(-i\nabla)b\|_{\frak{S}_q}\le(2\pi)^{-d/q}\|a\|_{L^q}\|b\|_{L^q}
$$
on the torus with $d=2$, $q=2$ and $a(k)=1/(|k|^2+m^2)$
(and with the same constant, since,
as the next Lemma shows, $\|a\|_{l^2(\mathbb{Z}_0^2)}\le\sqrt{\pi}/m$).
}
\end{rem}

\begin{lem}\label{L:T2}
For $m\ge0$
\begin{equation}\label{ineqT2}
F(m):= m^{2}\sum_{k\in\mathbb{Z}^2_0}
\frac{1}{(|k|^2+m^2)^2}<\pi.
\end{equation}
\end{lem}
\begin{proof}
We assume that $m\ge1$. We  show below that \eqref{ineqT2}
holds for $m\ge1$, which proves the Lemma, since
$F'(m)>0$ on $m\in(0,1]$ and $F$ is increasing on $m\in[0,1]$.

We use the Poisson summation formula (see,
e.\,g., \cite{S-W})
\begin{equation}\label{Poisson}
\sum_{m\in\mathbb{Z}^n}f(m/\mu)=
(2\pi)^{n/2}\mu^n
\sum_{m\in\mathbb{Z}^n}\widehat{f}(2\pi m \mu),
\end{equation}
where
$\mathcal{F}(f)(\xi)=\widehat{f}(\xi)=(2\pi)^{-n/2}\int_{\mathbb{R}^n}
f(x)e^{-i\xi x}dx$. For the function $f(x)=1/(1+|x|^2)^{-2}$,
$x\in\mathbb{R}^2$, with $\int_{\mathbb{R}^2}f(x)dx=\pi$, this gives
\begin{equation}
\label{Po_asymp}
\aligned
F(m)=
\frac1{m^2}\sum_{k\in\mathbb{Z}^2}f(k/m)-\frac1{m^2}f(0)=
\pi-\frac1{m^2}+
2\pi\sum_{k\in\mathbb{Z}^2_0}\widehat{f}(2\pi m k).
\endaligned
\end{equation}
Since $f$ is radial we have
$$
\widehat f(\xi)=\int_0^\infty \frac{J_0(|\xi|r)rdr}{(1+r^2)^2}=\frac {|\xi|}2K_1(|\xi|),
$$
where $K_1$ is the modified Bessel function of the second kind,
and where the second equality is formula $13.51\,(4)$ in \cite{Watson}.

Therefore we have to show that
$$
\sum_{k\in\mathbb{Z}^2_0}G(2\pi m |k|)<\frac1{m^2}, \quad
G(x)=\pi x K_1(x).
$$
Next, we use the estimate \cite{Bessel_approx}
$$
K_1(x)<\left(1+\frac1{2x}\right)\sqrt{\frac\pi{2x}}e^{-x}, \ \ x>0,
$$
which gives
$$
G(2\pi m|k|)<
\pi\left(\pi\sqrt{ m |k|}+\frac1{4\sqrt{m |k|}}\right)e^{-2\pi m|k|}.
$$
For the first term we use that
$$
\sqrt{x}e^{-a x}\le \frac{1}{\sqrt{2e a}}
$$
with $a=\frac12\pi m$ and $x=|k|$ (and keep three quarters of the negative exponent), while for the second term we
just replace $1/\sqrt{m|k|}$ by $1$, since $m\ge1$ and $k\ge1$. This gives
$$
G(2\pi m |k|)<\pi\left(\sqrt{\frac{\pi}{e}}e^{-3\pi m|k|/2}+\frac14 e^{-2\pi m |k|}\right).
$$
Furthermore, we use that $|k|\ge \frac1{\sqrt{2}}(|k_1|+|k_2|)$ and, therefore,
$$
G(2\pi m|k|)<\pi\left(\sqrt{\frac{\pi}{e}}e^{\frac{-3\pi m(|k_1|+|k_2|)}{2\sqrt{2}}}+
\frac14 e^{-\sqrt{2}\pi m (|k_1|+|k_2|)}\right).
$$
Thus, summing the geometric power series, we end up with
$$
\aligned
F(m)<\pi-\frac1{m^2}+\pi\sqrt{\frac{\pi}{e}}\left(\frac{4}{(e^{\frac{3\pi}{2\sqrt{2}}m} - 1)^2}
+\frac{4}{e^{\frac{3\pi}{2\sqrt{2}}m} - 1}\right)+\\
+\frac\pi 4\left(\frac{4}{(e^{\sqrt{2}\pi m} - 1)^2}
+\frac{4}{e^{\sqrt{2}\pi m} -1}\right).
\endaligned
$$
Introducing the functions
$$
\varphi(x):=\frac{x^2}{e^x-1},\ \ \psi(x):=\frac{x}{e^x-1}.
$$
we have to show that
$$
\aligned
\Psi(m):=4\sqrt{\frac{\pi}{e}}\left(\frac{2\sqrt{2}}{3\pi}\right)^2\left(\varphi\left(\frac{3\pi}{2\sqrt{2}}m\right)+
\psi\left(\frac{3\pi}{2\sqrt{2}}m\right)^2\right)+\\+\frac1{2\pi^2}
\left(\varphi\left(\sqrt{2}\pi m\right)+\psi\left(\sqrt{2}\pi m\right)^2\right)-\frac1{\pi}<0.
\endaligned
$$
Note that the function $\psi$ is obviously decreasing for all $m\ge0$,
so all terms involving $\psi$
 are decreasing. The function $\varphi(x)$ has a global maximum at
 $$
 x_0=1.5936\dots,
 $$
 and is decreasing when $x>x_0$. Since
 $$
 m_1:=\frac{2\sqrt{2}}{3\pi}x_0= 0.47824<1,\ \ m_2:=
 \frac1{\sqrt{2}\pi}x_0= 0.35868<1,
 $$
the function $\Psi(m)$ is decreasing for $m\ge1$ and it is sufficient to verify
 the inequality for $m=1$ only. Since
 $$
 \Psi(1)=-0.141093\dots<0,
 $$
inequality \eqref{ineqT2} is proved.
\end{proof}

\begin{rem}
{\rm
Since  $f(x)=1/(|x|^2+1)^2$ is analytic, its Fourier transform
decays exponentially and it follows from~\eqref{Po_asymp} that
$$
F(m)=
\pi-\frac1{m^2}+O\left(e^{-\mathrm{C}m}\right).
$$
This immediately gives that inequality \eqref{ineqT2}
holds on $[m_0,\infty)$ and we have to somehow specify
$m_0$  and then use numerical calculations to verify
\eqref{ineqT2} on a \emph{finite} interval $(0,m_0)$
(precisely this is the case of a similar estimates in~\cite{ZIL-JFA}).
Lemma~\ref{L:T2} is one of the few examples
 when this can be done purely analytically.
The key points are of course the explicit
formula for the Fourier transform and  uniform founds for~$K_1$.
}
\end{rem}

% Optional acknowledgments should appear right before the bibliography.
\begin{acknowledgments}
This work was done with financial support from
the Moscow Centre of Fundamental and Applied Mathematics at the
 Keldysh Institute of Applied Mathematics (project
25-05-01).
\end{acknowledgments}

\end{document}